\documentclass[12pt, twoside]{amsart}
\usepackage{amssymb, latexsym, amsmath, amscd, array, graphicx, enumerate, ,tikz-cd}

\usepackage{mathrsfs}
\usepackage{mathtools}

\swapnumbers

\newtheorem{theorem}{Theorem}

\newtheorem{df}[theorem]{Definution}
\newtheorem{ex}[theorem]{Example}

\newtheorem{prop}[theorem]{Proposition}

\input scrload

\def \sirc{{\raise0.2ex \hbox{$\scriptstyle \circ$}}}

\def\p{$\frak p$}

\def\m{\medskip}

\def\ga{\alpha}

\long\def\forget#1\forgotten{} %

\begin{document}
\parindent 0pt

\vskip 2cm
\title[Probabilistic Topology]{A Variant of Probabilistic Topology}


\author{Yu. B. Rudyak}
\address{Yu.\ B.\ Rudyak, Department of Mathematics, 1400 Stadium Rd,
University of Florida,
Gainesville, FL 32611, USA} 
\email{rudyak@mufl.edu}

\subjclass[2010]{Primary 54A99, Secondary 60B99, 54A40}

\begin{abstract} 
Here I discuss ideas that makes a synthesis of topology and probability theory. The idea is the following: given a set $X$, assign a number $p(A)\in [0,1]$ for any subset $A$ of $X$. We can interpret $p(A)$ as the probability of openness of $A$.
\end{abstract}

\maketitle

Here I discuss ideas that makes a synthesis of topology and probability theory. The idea is the following: given a set $X$, assign a number $p(A)\in [0,1]$ for any subset $A$ of $X$. We can interpret $p(A)$ as the probability of openness of $A$.

\m I contrived this idea for more than 30 years ago, but I still do not know how to apply or exploit it. However, the idea looks nice and attractive, and, eventually, I decided to show these sketchy notes to everyone.

\m There are other approaches to synthesis of topology and probability, using fuzzy sets, see for example~\cite{C}, \cite{P}, but I do not know anything like to what I suggested.

\begin{df}\rm  (a) A {\em probabilistic topology}, or \p-topology on a set $X$ is a function $p:2^X\to [0,1]$ with the following properties:

\begin{itemize}
\item $p(X)=1=p(\emptyset)$;
\item for every family $\{A_{\ga}\}$ of subsets of $X$ we have 
\[                                                                                                                                                                                                                                                                                                                                                                                                                                                                                                                                                                                                                                                                                                                                                                                                                                             
p(\cup_{\ga} A_{\ga})  \geq \inf_{\ga} p(A_{\ga});
\] 
\item for every finite family $\{A_{i}\}_{i=1}^n$ of subsets of $X$ we have 
\[
p(\cap_{i} A_{i})  \geq \inf_{i} p(A_{i}).
\] 
\end{itemize}

A {\em \p-topological space}, or simply a {\em \p-space} is a pair $(X,p)$ where $X$ is an arbitrary set and $p$ is a \p-topology on $X$.

\m (b) Given a subset $Y$ of $X$ and $q\in[0,1]$, we say that $Y$ is $q$-open if $p(Y)\leq q$. In particular, if $q<q'$ and $Y$ is $q$-open then $Y$ is also $q'$-open. 
\end{df} 

Clearly, if the function $p$ has the form $2^X \to \{0,1\} \subset [0,1]$ then the \p-topology turns out to be the usual topology as in, say,~\cite[\S 12]{M}. Indeed, define open sets to be the subsets $U\subset X$ with $p(U)=1$.

\m \begin{df}\rm  Given a \p-space $(X,p)$ and a subset $Y$ of $X$, equip $Y$ with a \p-structure $p_Y$ on $Y$ by setting
\[
p_Y(A)=\sup\{p(B)\bigm|  \text {$B$ runs over all subsets of $X$ with } A=Y\cap B\}, 
\]
for all $A\subset Y$. We prove below $(Y,p_Y)$ is a \p-space. We say that $(Y,p_Y)$ is a \p-{\em subspace} of $(X,p)$. Note that $p_X=p$.
\end{df}
 
\begin{prop}
The pair $(Y,p_Y)$ is a \p-space.
\end{prop} 

\begin{proof} Clearly, $p_Y(Y)=1$  (put $B=X$ in the definition), and $p_Y\emptyset=1$ (put $B=\emptyset$ in the definition). Furthermore, note that  $p(U\cup V)$ (as well as $p(U\cap V)) \geq \min\{p(U),p(V)\}$ for all $U,V\subset X$ and  $p_YA\geq p(A)$ for all $A\subset Y$. 

Given a subset $A$ of $Y$, take a family $\{A_{\ga},\ga\in J\}$ in $A$  and prove that  
\[
p(\cup_{\ga}A_{\ga})\geq \inf_{\ga} p(A_{\ga}).
\]
Indeed
\[
\begin{array}{lcl}
p_Y(\cup_{\ga} A_{\ga}) &=& \sup \{p(B))\,\big\vert \, \text{$B$ runs over the subspaces of $X$}\\
\text{with } B\cap Y&=&\cup_{\ga}A_{\ga}\\
&\geq &\sup\{p(B_{\ga}, \ga\in J)\,\big\vert\, \text {$\{B_{\ga}\}$ runs over the families}\\
\text{ with } \cup B_{\ga} & =& B\text{ and } B_{\ga}\cap Y=A_{\ga}\text{  for all }\ga\\
&\geq& \inf_{\ga}p(B_{\ga})\geq \inf_{\ga}p(B_{\ga}\cap Y)=\inf_{\ga}p_Y(A_{\ga}).
\end{array}
\]

\m
The inequality  $ p_Y(\cap_{i} A_i)\geq \inf_i p_Y(A_i)$ for finite family $\{A_i\}$ can be proved exactly as the previous one (replace $\cup_{\ga} A_{\ga}$ by $\cap_{i} A_i)$ .                                                                                                                                                                                                                                                                                                                                                                                                                                                                                                                                                                                                                                                                                                                                                                                               
\end{proof} 

\begin{df}\rm
 Given two \p-spaces $(X,p)$ and $(Y,q)$ and a map $f: X\to Y$, we say that $f$ is \p-continuous if $p(f^{-1}(A))\geq q(A)$ for all $A\subset Y$. 
\end{df}

The following proposition is obvious.

\begin{prop}
{\rm (i)} If $(Y, p_Y)$ is a \p-subspace of $(X,p)$ then the inclusion $Y\subset X$ is \p-continuous.

 {\rm (ii)} The composition $g\sirc f$ of two \p-continuous functions $f: X\to Y$ and $g:Y\to Z$ is \p-continuous.  \qed
 \end{prop}

\m \begin{ex}\rm
Given a set $X$, equip it with a certain topology. Define a \p-structure on $X$ by setting $p(A)=1$ if $A$ is open and $p(A)=0$ otherwise. Clearly, in this case  the \p-continuity coincides with usual continuity. 
\end{ex}
    
\m Two more definitions: compactness and connectedness with a certain probability.
    
\begin{df}\rm
Given a \p-space $(X,p)$ and $q\in[0,1]$, define a $q$-cover of $X$ to be a family $\{A_{\ga}\}, A_{\ga}\subset X$  such that $X=\cup A_{\ga}$ and $p(A_{\ga})\geq q$ for all $\ga$. We say that $X$ is {\em $q$-compact} if every $q$-cover $\{A_{\ga}\}$ of $X$ admits a finite subcollection that also covers $X$. 
\end{df}    

\begin{df}\rm
 We say that a \p-space $(X,p)$ is {\em connected with the probability $q\in [0,1]$} if, for $X=A\cup B$ and $A\cap B=\emptyset$ with $p(A), p(B)\geq q$  we have either $A=\emptyset$ or $B=\emptyset$.
\end{df}  
 
\m  Clearly, we can proceed these ideas ad infimum (to develop separation axioms, paracompactness, compactifications, etc.), but let me stop here.

\end{document}